\documentclass{amsart}

\usepackage{all2018}
\usepackage{pxfonts}
\renewcommand{\sgn}{\operatorname{sgn}}
\newcommand{\Tube}{\operatorname{Tube}}
\newcommand{\cJ}{\mathcal{J}}
\newcommand{\corr}{\mathrm{corr}}
\newcommand{\Vol}{\mathrm{Vol}}

\begin{document}

\title[Partial correlation hypersurfaces]{Partial correlation hypersurfaces\\ in Gaussian graphical models}

\author{Jan Draisma}
\address{Universit\"at Bern, Mathematisches Institut,
Sidlerstrasse 5,
3012 Bern, Switzerland; and Eindhoven University of Technology,
Department of Mathematics and Computer Science, P.O.~Box 513, 5600 MB
Eindhoven}
\email{jan.draisma@math.unibe.ch}
\thanks{The author is partially supported by a Vici grant from the Netherlands
Organisation for Scientific Research (NWO)}

\begin{abstract}
We derive a combinatorial sufficient condition for a partial correlation
hypersurface in the parameter space of a directed Gaussian graphical model
to be nonsingular, and speculate on whether this condition can be used in
algorithms for learning the graph. Since the condition is fulfilled in the
case of a complete DAG on any number of vertices, the result implies an
affirmative answer to a question raised by Lin-Uhler-Sturmfels-B\"uhlmann.
\end{abstract}

\maketitle

\section{Introduction} \label{sec:Introduction}

\subsection*{DAGs}
Let $G$ be a directed, acyclic graph (DAG) with vertex set $V$ and edge
set $D \subseteq \{(i,j) \in V^2 \mid i \neq j\}$. We write
$i \to j$ if $(i,j) \in D$ and $i \not \to j$ otherwise. A path in $G$ from $i$ to $j$
of length $k$ is a sequence $(i=i_0,i_1,\ldots,i_k=j)$ with $i_{l} \to
i_{l+1}$ for all $l=0,\ldots,k-1$; we allow $k=0$. If there exists a path
from $i$ to $j$ of length at least $1$ we say that $j$ is below $i$.

\subsection*{Directed Gaussian graphical models} 
We follow \cite[Page 87]{Drton09}. Associated to $G$ is the directed
graphical model for jointly Gaussian random variables $X_i,\ i \in V$
related by
\[ X_j=\sum_{i: i \to j} a_{ij} X_i + \epsilon_j\]
where the vector $\epsilon \sim \mathcal{N}(0,I)$ and where the
$a_{ij} \in \RR$ are the parameters of the model. The
vector $X=(X_j)_{j \in V}$ satisfies
\[ (I-A)^T X=\epsilon \]
where $A$ is the matrix with $(i,j)$-entry $a_{ij}$ if $i
\to j$ and $0$ otherwise. Therefore $X\sim \mathcal{N}(0,\Sigma)$ where
\[ \Sigma=\Sigma(A)=(I-A)^{-T} (I-A)^{-1}. \]
Note that, since $A$ is nilpotent, this is a matrix whose entries
are polynomials in the parameters $a_{ij}, i \to j$. For subsets
$I,J \subseteq V$ we write $\Sigma[I,J]$ for $I \times J$-submatrix
$(\sigma_{ij})_{i \in I,j \in J}$ of $\Sigma$, and we use notation
such as $I+i_0-s:=I \cup \{i_0\} \setminus \{s\}$.

\subsection*{Partial correlation hypersurfaces}
Let $i_0,j_0 \in V$ be distinct and $S \subseteq V \setminus
\{i_0,j_0\}$. In \cite{Uhler14} the {\em partial correlation hypersurface}
$H_f \subseteq \RR^D$ is defined as the zero locus of the
polynomial
\[ f:=\det(\Sigma[S+i_0,S+j_0]); \]
the expression $\corr(i_0,j_0|S):=f/\sqrt{\det(\Sigma[S+i_0,S+i_0])
\det(\Sigma[S+j_0,S+j_0])}$ is the partial correlation of $i_0$ and $j_0$
given $S$.

So the vanishing of $f$ is equivalent to the statement that $i_0,j_0$ are
conditionally independent given $S$. We assume that $f$ is not identically
zero on $\RR^D$.  This is equivalent to the statement that $S$ does not
{\em d-separate} $i_0$ and $j_0$ in $G$ \cite[\S 2.3.4]{Spirtes01};
the trek system expansion of Section~\ref{sec:Background} yields an
equivalent combinatorial characterisation.

The key motivation in \cite{Uhler14} for studying $H_f$ is that the
behaviour for $\lambda \to 0$ of the volume (relative to some probability
measure) of $\Tube(\lambda):=\{a \in \RR^D : |\corr(i_0,j_0,S)| \leq
\lambda\}$ is related to the singularities of $H_f$. This volume scales
linearly with $\lambda$ if $H_f$ is nonsingular but can be superlinear
otherwise---whence the study of the real log-canonical threshold of
$H_f$ in \cite{Uhler14}. The parameter values $a$ in $\Tube(\lambda)$
correspond to probability distributions that are not {\em
$\lambda$-strongly-faithful} to $G$---distributions where
the PC algorithm for learning $G$ might fail. So it is useful to know
criteria for nonsingularity of $H_f$.

\subsection*{Main result}
We will establish the following criterion for nonsingularity of $H_f$;
the same applies when the $\epsilon_i$ have unequal variances 
(Proposition~\ref{prop:unequal}).

\begin{thm}
Assume that $i_0 \to j_0$ and that for all $s \in S$ below $j_0$ we have
$i_0 \to s$. Then $H_f$ is nonsingular.
\end{thm}

\begin{cor}
If $G$ is the DAG on $\{1,\ldots,n\}$ with $i \to j$ if and only if $i<j$, then $H_f$
is nonsingular, independently of the choice of $i_0,j_0,S$.
\end{cor}

For $n \leq 6$ this is \cite[Theorem 4.1]{Uhler14}, which was established
there by extensive computer calculations showing that some power of
$\det \Sigma[S+i_0+j_0,S+i_0+j_0]$ lies in the ideal generated by $f$
and its partial derivatives. Since $\Sigma[S+i_0+j_0,S+i_0+j_0]$ is
positive definite and hence has a nonzero determinant for all (real)
values of the parameters, this shows that the (real) 
common vanishing locus of $f$ and its derivatives is empty.

We will follow a similar approach, except
that we consider the principal submatrix $\det \Sigma[S+i_0,S+i_0]$,
no power is needed, and indeed not $f$ but only some of its partial
derivatives are needed.

\subsection*{Organisation}

In Section~\ref{sec:Background} we review the expansion of subdeterminants
of $\Sigma$ in terms of trek systems without sided intersection
\cite{Draisma09d}. In Section~\ref{sec:Proof} we use this to prove
the theorem, and we conclude with a brief discussion in 
Section~\ref{sec:Open}.

\section{Background} \label{sec:Background}

\subsection*{The trek rule}
We recall results from~\cite{Draisma09d}.  Suppose we allow the variances
of the $\epsilon_i$ to be distinct, rather than all equal to $1$ as
above. In that case, the covariance matrix $\Sigma$ becomes
\[ \Sigma=(I-A)^{-T} \Omega (I-A)^{-1} \]
where $\Omega$ is the diagonal matrix with the covariances of the
$\epsilon_i$ on the diagonal. Using the geometric series for $(I-A)^{-1}$
we find that
\[ \sigma_{ij}=\sum_{t:i \to j} w(t) \]
where the sum is over all {\em treks} from $i$ to $j$ as in
the following definition.

\begin{de}
A {\em trek} $t$ in $G$ is a pair $(P_U,P_D)$ of paths in $G$ that
start at the same vertex $m$, the {\em top} of the trek. The paths
$P_U,P_D$ are called the {\em up part} and the {\em down part} of $t$,
respectively. If $i_0$ is the last vertex of $P_U$ and $j_0$ is the last
vertex of $P_D$, then we call $t$ is a trek {\em from $i_0$ to $j_0$},
$i_0$ the {\em starting} vertex of $t$, and $j_0$ the {\em end} vertex
of $t$. The {\em weight} of $t$ equals
\[ w(t):= \left(\prod_{(i,j) \text{ in } P_U} a_{ij} \right) \cdot \omega_m \cdot
\left(\prod_{(i,j) \text{ in } P_D} a_{ij} \right). \]
We allow one or both of $P_U,P_D$ to have length $0$, in which case the
corresponding factor(s) above is (are) $1$.
\end{de}

The terminology derives from an informal interpretation of a trek as
traversing $P_U$ upwards from $i_0$ (i.e., against the direction of
its edges in $G$) and then traversing $P_D$ downwards to $j_0$. In
slightly different terms, the {\em trek rule} above goes back at least
to \cite{Wright34}.

\subsection*{Trek system expansion}
Equip $V$ with an arbitrary linear order. Then for $I,J \subseteq V$
of equal cardinality and $\pi:I \to J$ we define $\sgn(\pi)$ as $(-1)$
to the power the number of {\em crossings}: pairs $(i_1,i_2) \in I^2$
with $i_1<i_2$ but $\pi(i_1)>\pi(i_2)$.

\begin{de}
Let $I,J \subseteq V$ with $|I|=|J|=k$. A {\em trek system} $T$ from $I$
to $J$ is a set of treks $\{t_1,\ldots,t_k\}$ such that $I$ is precisely
the set of starting vertices of the $t_l$ and $J$ is precisely the set
of end vertices of the $t_l$. We write $T:I \to J$.  The map $\pi:I \to
J$ that sends the starting vertex of each trek to its end vertex is a
bijection, and we define the sign of $T$ as $\sgn(T):=\sgn(\pi)$. The
weight of $T$ is $w(T):=\prod_{l=1}^k w(t_l)$.
\end{de}

\begin{de}
A {\em sided intersection} between treks $t$ and $t'$ is a vertex where
either the up parts of $t$ and $t'$ meet or the down parts of $t$ and $t'$
meet. We say that a trek system has no sided intersections if there is
no sided intersection between any two of its treks.
\end{de}

We have the following formula for subdeterminants of $\Sigma$.

\begin{prop}[\cite{Draisma09d}]
For $I,J \subseteq V$ of the same cardinality we have
\begin{equation} \det \Sigma[I,J]=\sum_{T:I \to J \text{ without sided
intersections}} \sgn(T) \wt(T).  \label{eq:nosided} \tag{*}
\end{equation}
\end{prop}

The proof is an application of tail swapping as in the classical
Lindstr\"om-Gessel-Viennot Lemma \cite{Gessel85}. We will see
another instance of tail swapping in Section~\ref{sec:Proof}.
In \cite{Draisma09d} the proposition is used to give
a combinatorial criterion, generalising d-sepa\-ra\-tion, for the
determinant to be identically zero on $\RR^D \times \RR_{>0}^V$. Furthermore,
in \cite{Draisma12h} it is shown that the sum above is cancellation-free:
if two trek systems $I \to J$ have the same weight, then they have the
same sign. Moreover, it is shown there that the coefficient of each
monomial is plus or minus a power of $2$.

All of these results---the formula \eqref{eq:nosided} of
course, but also the cancellation-freeness and the power-of-two
phenomenon---persist when we specialise $\Omega$ to the identity matrix,
as we did in Section~\ref{sec:Introduction} and as we do again in
Section~\ref{sec:Proof}. Indeed, if $T: I \to J$ is a trek system without
sided interaction, then the tops of the treks in $T$ can be recovered
from the specialisation of $w(T)$ as follows: $m$ is a top if and only if
either 
\begin{enumerate}
\item at least one $a_{mj}$ appears in $w(T)$ and no $a_{im}$ appears
in $w(T)$; or else 
\item $m \in I \cap J$ and $w(T)$ contains no $a_{mj}$
and no $a_{im}$ (then some trek is $((m),(m))$). 
\end{enumerate}

\subsection*{Action by diagonal matrices}
Let $d=\diag((d_i)_{i \in V})$ where the $d_i$ are in $\RR_{>0}$. Then
\[ d \Sigma d=
(d(I-A)^{-T}d^{-1}) \cdot (d \Omega d) \cdot (d^{-1}(I-A)^{-1}d)=
(I-A')^{-T} \Omega' (I-A')^{-1} \]
where $\Omega'=d \Omega d$ and where $A'=d^{-1}Ad$ has the same zero
pattern as $A$. Hence, the group $(\RR_{>0})^V$ acts on the parameter
space $\RR^D \times \RR_{>0}^V$ and on the space of covariance matrices in such a manner that
the map $(a,\omega) \mapsto \Sigma$ is equivariant. This implies that
for any $I, J \subseteq V$ of equal cardinality the hypersurface in
$\RR^D \times \RR_{>0}^V$ defined by $\det \Sigma[I,J]=0$ is stable
under this action. 

Alternatively, this can be read off from \eqref{eq:nosided}:
scaling each $a_{ij}$ with $d_i^{-1} d_j$ and $\omega_m$ with $d_m^2$,
the weight of each trek from a vertex $i \in I$ to a vertex $j \in J$
gets scaled by $d_i d_j$, and therefore $\det \Sigma[I,J]$ scales
with $\left(\prod_{i \in I} d_i \right) \cdot \left( \prod_{j \in J}
d_j \right)$.

Define $f_{\Omega}:=\det \Sigma[I,J]$ and let $f$ be obtained from
$f_{\Omega}$ by specialising $\Omega$ to the identity matrix. Let $H_f$
be the hypersurface in $\RR^D$ defined by $f$ and let $H_{f_\Omega}$
be the hypersurface defined by $f_\Omega$ in $\RR^D \times \RR_{>0}^V$.

\begin{prop} \label{prop:unequal}
As real algebraic varieties, $H_{f_\Omega}$ is isomorphic to $H_f \times
\RR_{>0}^V$. In particular, $H_{f_\Omega}$ is nonsingular if and only
if $H_f$ is.
\end{prop}

\begin{proof}
By the discussion above, the map
\[ (a,d) \mapsto \left((a_{ij} \cdot \frac{d_j}{d_i})_{i \to j},
(d_m^2)_m\right) \]
maps $H_f \times \RR_{>0}^V$ into $H_{f_\Omega}$. The inverse is given by 
\[ (a',\omega) \mapsto \left((a'_{ij} \cdot
\frac{\sqrt{\omega_i}}{\sqrt{\omega_j}})_{i \to j},
(\sqrt{\omega_m})_m \right). \]
Both maps are morphisms of real algebraic varieties.
\end{proof}

\section{Proof of the theorem} \label{sec:Proof}

We retain the notation of Section~\ref{sec:Introduction}; in particular,
$\epsilon \sim \mathcal{N}(0,I)$, $f=\det \Sigma[S+i_0,S+j_0]$ and $H_f
\subseteq \RR^D$ is the hypersurface defined by $f$. In this section,
we treat the $a_{ij}$ as variables and our computations take place in the
polynomial ring $\RR[a_{ij} \mid (i,j) \in D]$. Let $\cJ$ be the ideal in
this ring generated by all partial derivatives $\pafg[f]{a_{ij}}$ of $f$.

\begin{lm} \label{lm:sj}
For $s \in S$ and $j \in V$ with $s \to j$ the variable $a_{sj}$ does
not appear in $f$.
\end{lm}

\begin{proof}
Let $T:S+i_0 \to S+j_0$ be a trek system without sided intersection.
If the arrow $s \to j$ were used in the up (respectively, down)
part of some trek $t$ in $T$, then $t$ would have a sided intersection
with the trek starting (respectively, ending) at $s$. So that arrow is not 
used and the conclusion follows from \eqref{eq:nosided}.
\end{proof}

As a consequence, in the remaining discussion we may and will replace $D$
by $D \setminus S \times V$, so that {\em $G$ has no arrows going out
of $S$}.

\begin{lm} \label{lm:i0s}
Suppose that $G$ has no outgoing arrows from elements of $S$.
For $s \in S$ with $i_0 \to s$ the variable $a_{i_0s}$ appears
at most linearly in $f$ and its coefficient equals $\pm \det
\Sigma[S+i_0,S+j_0-s+i_0]$. In particular, $\det \Sigma[S+i_0,S+j_0-s+i_0]
\in \cJ$.
\end{lm}

\begin{proof}
If a trek $t$ in a trek system $T:S+i_0 \to S+j_0$ without sided
intersection uses the edge $i_0 \to s$, then it does so in its down
part---indeed, in its up part it would yield a sided intersection with
the trek starting at $i_0$.

In particular, the variable $a_{i_0s}$ appears only linearly in $f$. 
Furthermore, $t$ ends in $s$, or else $t$ would have a sided intersection
with the trek ending at $s$. So if we remove from $t$ the arrow $i_0
\to s$, then we obtain a trek system $T':S+i_0 \to S+j_0-s+i_0$ without
sided intersection (Figure~\ref{fig:i0s}).

\begin{figure}
\begin{center}
\includegraphics[scale=.7]{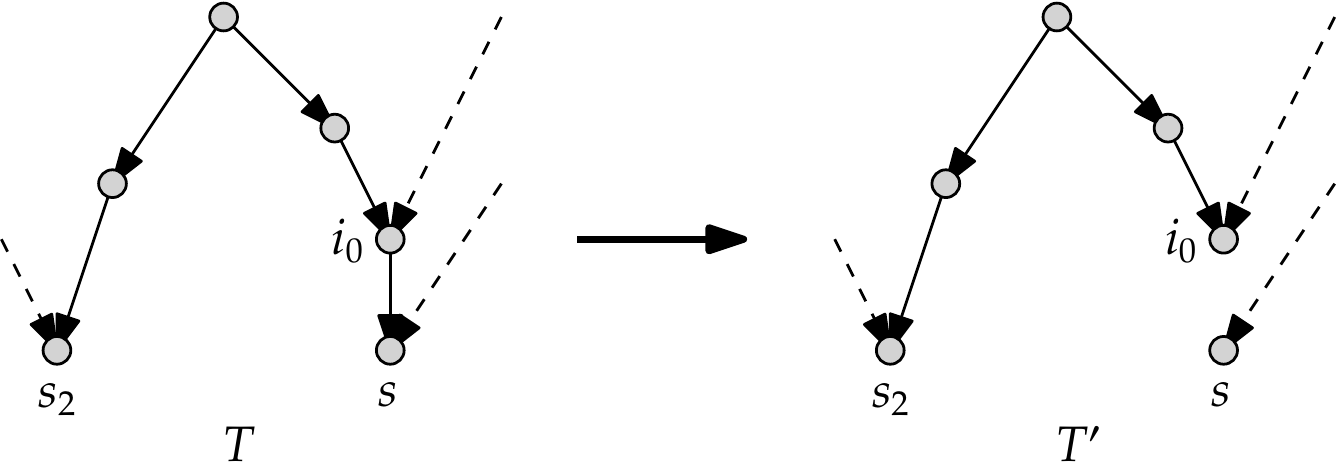}
\end{center}
\caption{Proof of Lemma~\ref{lm:i0s}. We suggestively draw the
arrows in up parts of treks as pointing in the south-west direction
and arrows in down parts as pointing in the south-east direction---of
course, this is not always possible!}
\label{fig:i0s}
\end{figure}

Conversely, if we have any trek system $T':S+i_0 \to S+j_0-s+i_0$ without
sided intersection, then no trek in it passes through $s$ on its way down,
because $s$ has no outgoing arrows. Hence, adding the arrow $i_0 \to s$
to the trek $t'$ in $T'$ ending in $i_0$ yields a trek system $S+i_0
\to S+j_0$ without sided intersection.

Hence the map $T \mapsto T'$ gives a bijection between the terms in (the
trek system expansion of) $f$ divisible by $a_{i_0s}$ and the terms in
$\det \Sigma[S+i_0,S+j_0+i_0-s]$. Furthermore, $\sgn(T)$ equals $\pm
\sgn(T')$, where the sign is the sign of the bijection $S+j_0-s+i_0
\to S+j_0$ that is the identity on $S+j_0-s$ and sends $i_0$ to $s$;
in particular, this sign does not depend on $T$.  \end{proof}

\begin{lm} \label{lm:i0j0}
Assume that $i_0 \to j_0$. The variable $a_{i_0j_0}$ appears
at most linearly in $f$ and its coefficient equals $\pm(\det \Sigma[S+i_0,S+i_0]-g)$
where 
\begin{equation} 
g=\sum_{T'':S+i_0 \to S+i_0} \sgn(T'') w(T'')
\label{eq:i0j0} \tag{**}
\end{equation}
is the sum over all trek systems $T'':S+i_0 \to S+i_0$ without sided
intersection {\em of which one trek contains $j_0$ in its
down part}. In particular, $\det \Sigma[S+i_0,S+i_0] - g \in
\cJ$.
\end{lm}

\begin{proof}
If a trek $t$ in a trek system $T:S+i_0 \to S+j_0$ without sided
intersection uses the edge $i_0 \to j_0$, then it does so on its way
down: on its way up it would yield a sided intersection with the trek
starting at $i_0$. In particular, the variable $a_{i_0j_0}$
appears only linearly in $f$. 

Furthermore, $t$ ends in $j_0$, or else it would have a sided intersection
with the trek ending at $j_0$. So if we remove from $t$ the arrow $i_0
\to j_0$, then we obtain a trek system $T'':S+i_0 \to S+i_0$ without sided
intersection (Figure~\ref{fig:i0j0}). Also, $\sgn(T)$ equals $\sgn(T'')$
times the sign of the bijection $S+i_0 \to S+j_0$ that is the identity
on $S$ and maps $i_0$ to $j_0$; this will determines the sign $\pm$
in the lemma. 

\begin{figure}
\begin{center}
\includegraphics[scale=.7]{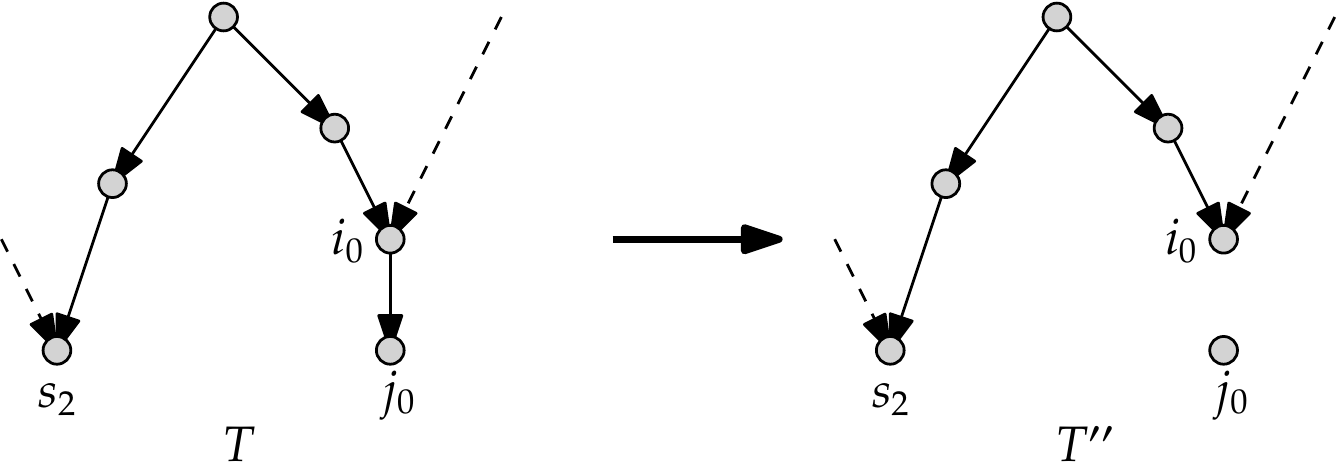}
\end{center}
\caption{Proof of Lemma~\ref{lm:i0j0}.}
\label{fig:i0j0}
\end{figure}

Conversely, given a trek system $T'':S+i_0 \to S+i_0$ without
sided intersection, we may try and add the arrow $i_0 \to j_0$
to the trek ending in $i_0$. The resulting trek system has no sided
intersection if and only if no trek of $T''$ passes $j_0$ on its way
down. The remaining $T''$ must be therefore be subtracted as in the lemma.
\end{proof}

For $s \in S$ under $j_0$ define 
\[ p_{j_0,s}:=\sum_{P:j_0 \to s} w(P), \]
the sum of the weights of all directed paths in $G$ from $j_0$ to $s$.

\begin{lm} \label{lm:id}
The element $g$ from \eqref{eq:i0j0} satisfies
\[ g=\sum_{s \in S \text{ under } j_0} \sgn(\pi_s) \det \Sigma[S+i_0,S+i_0-s+j_0]
\cdot p_{j_0,s} \]
where $\pi_s:S+i_0-s+j_0 \to S+i_0$ is the
identity on $S+i_0-s$ and sends $j_0$ to $s$.
\end{lm}

\begin{proof}
Let $T':S+i_0 \to S+i_0-s+j_0$ be a trek system without sided intersection
and let $t'$ be the trek of $T'$ ending in $j_0$. Appending to $t'$ any
path from $j_0$ down to $s$ yields a trek system $T'':S+i_0 \to S+i_0$
with sign $\sgn(T'')=\sgn(T') \sgn(\pi_s)$. In this manner, precisely 
those trek systems $T'':S+i_0  \to S+i_0$ arise for which 
\begin{enumerate}
\item a unique trek $t''$ of $T''$ passes $j_0$ on its way down, and 
\item every sided intersection of $T''$ is between $t''$ and some other
trek of $T''$ on their way down, and happens at a vertex below $j_0$. 
\end{enumerate}
So the left-hand side of the equation in the lemma equals $\sum_{T'':S+i_0
\to S+i_0} \sgn(T'') w(T'')$ where $W''$ runs over the trek systems with
properties (1) and (2). The right-hand side is the sub-sum over all $T''$
without any sided intersection. We construct a sign-changing involution
on the remaining $T''$, as follows. 

Let $k$ be the {\em lowest} vertex on the down part of $t''$ that lies on
the down part of some other trek $u'' \neq t''$ of $T''$. Swapping the
parts of $t''$ and $u''$ below $k$ yields treks $t'''$ and $u'''$ that
still meet at $k$. Let $T'''$ be the trek system obtained from $T''$ by
replacing $t''$ with $t'''$ and $u''$ with $u'''$ (Figure~\ref{fig:id}).

\begin{figure}
\begin{center}
\includegraphics[scale=.7]{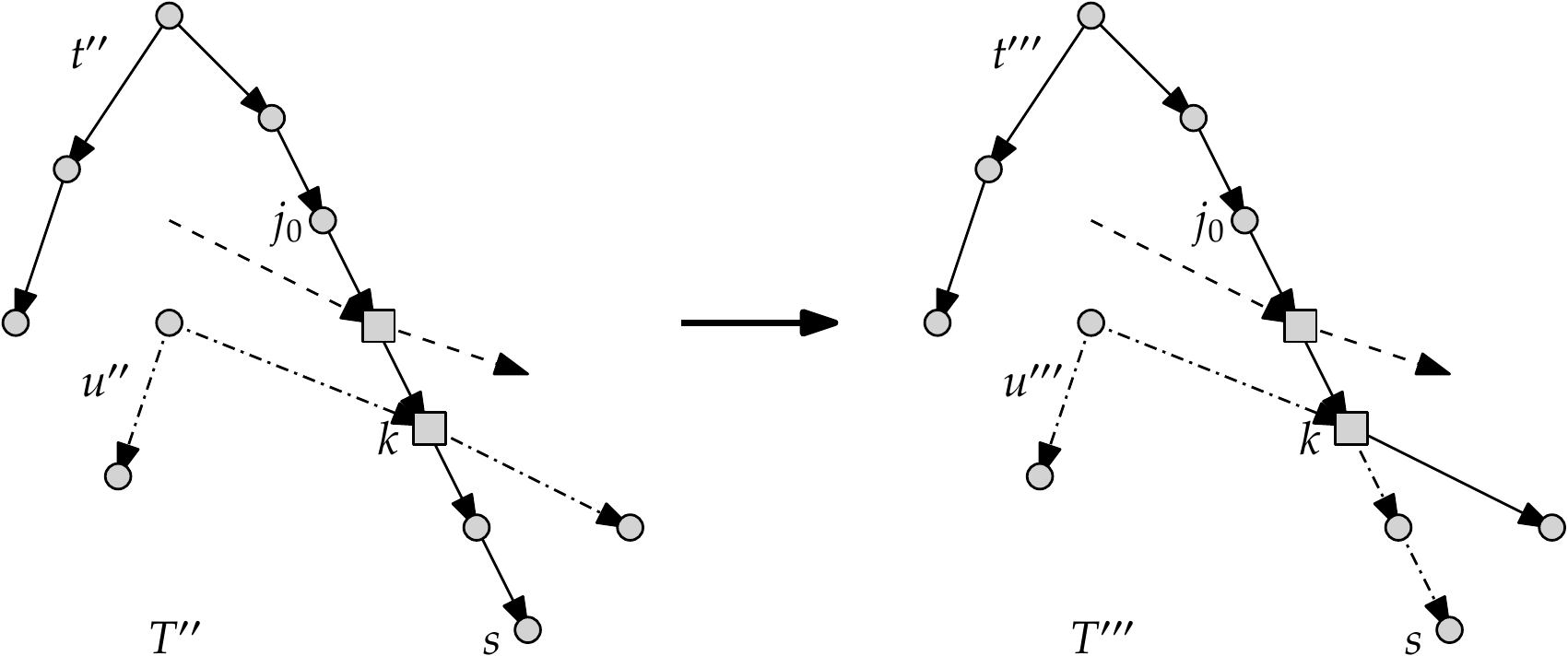}
\caption{The tail swapping argument of Lemma~\ref{lm:id}. The sided
intersections of $t''$ with other treks are depicted as square vertices.}
\label{fig:id}
\end{center}
\end{figure}

The trek system $T'''$ satisfies (1): $t'''$ is its unique trek that
passes $j_0$ on its way down. As for (2): the sided intersections
between $t'''$ and other treks are precisely the sided intersections
between $t''$ and other treks, so they happen below $j_0$. Furthermore,
$u'''$ cannot have sided intersections with treks other than $t'''$,
because those would have come from a sided intersection between $t'''$
and another trek happening below $k$---this is where the choice of $k$
matters. Furthermore, $T''' \setminus \{t''',u'''\}=T'' \setminus
\{t'',u''\}$, so there are no sided intersections between these
treks. This shows that $T'''$ satisfies (2). Also, the map $T'' \to
T'''$ is an involution, since $k$ is the last intersection of the
down part of $t'''$ with any down part of a trek in $T'''$. Since
$\sgn(T''')=-\sgn(T'')$, this shows that the terms on the left-hand
side that do not appear in the right-hand side cancel out. 
\end{proof}

\begin{proof}[Proof of the theorem]
We claim that the zero set of $\cJ$ in
$\RR^D$ is empty. By Lemma~\ref{lm:sj} we may delete from $G$ all
outgoing arrows from elements of $S$ without changing $f$.
Since $i_0 \to j_0$, by Lemma~\ref{lm:i0j0} we have
$\det \Sigma[S+i_0,S+i_0]-g \in \cJ$. The identity in Lemma~\ref{lm:id}
expresses $g$ as a linear combination of the determinants in
Lemma~\ref{lm:i0s} where $s$ runs over the elements of $S$ below
$j_0$. By assumption, for each of these $s$ we have $i_0 \to s$, so
Lemma~\ref{lm:i0s} implies that $g \in \cJ$. Hence $\det \Sigma[S+i_0,S+i_0]
\in \cJ$. But for any set of real parameters $a \in \RR^D$ the matrix
$\Sigma[S+i_0,S+i_0]$ is positive definite, hence has a nonzero
determinant. This proves the claim.
\end{proof}

\section{A modest implication for the PC algorithm} \label{sec:Open}

In the edge-removal part of the PC algorithm \cite{Spirtes01} for learning
$G$, in each step we have an undirected graph $H$ whose edge set, if no
error has occurred so far, contains that of $G$. Using the sample
covariance matrix, a partial correlation
$\corr(i_0,j_0|S)$ is then computed
for some triple $i_0,j_0,S$ such that there is an edge $i_0-j_0$ in $H$
and such that $S$ is contained in the $H$-neighbours of $i_0$ or in
the $H$-neighbours of $j_0$. Before this step all partial correlations
with sets $S'$ of cardinality smaller than that of $S$ have already been
checked. If the absolute value of the partial correlation is less than
some prescribed $\lambda$, then the edge $i_0-j_0$ is removed from $H$.

Our theorem suggests that it might be advantageous to perform this
check first for sets $S$ contained in the {\em intersection} of the
neighbourhoods of $i_0$ and $j_0$ in $H$. Then, {\em if all the edges
between $i_0,j_0,S$ present in $H$ are also present in the DAG $G$} (with
some orientation), one readily checks that the conditions of the theorem
are satisfied. Hence the volume of $\Tube(\lambda)$ is proportional to
$\lambda$, and the region in the parameter space $\RR^D$ of $G$ where
we would erroneously delete $i_0-j_0$ in this step is small.

There are two obvious issues with this. First, in general it will not
suffice to check $S$ in the intersection of the neighbourhoods of $i_0$
and $j_0$. And second, the condition that all of those edges are indeed
present in $G$ is rather strong. To make better use of our theorem, one
might want to develop a version of the PC algorithm where orientation
steps are intertwined with the edge-deletion steps.

We conclude this paper with two examples.

\begin{figure}
\begin{center}
\includegraphics[scale=.7]{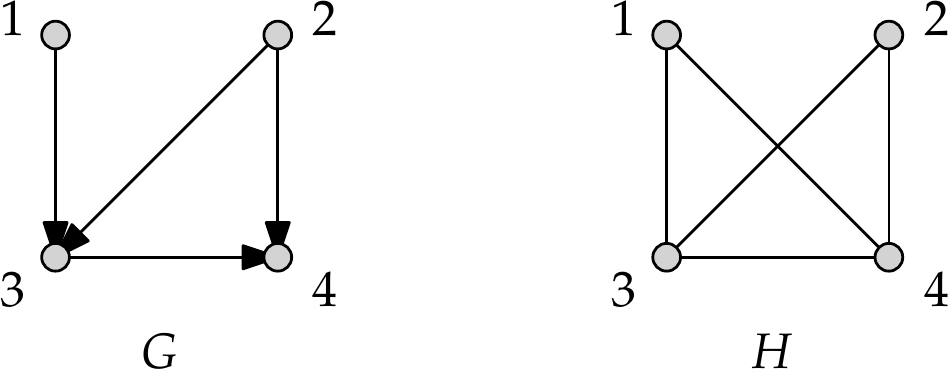}
\label{fig:ex}
\caption{The graphs in Example~\ref{ex:graph}.}
\end{center}
\end{figure}

\begin{ex} \label{ex:graph}
To see that singular partial correlation hypersurfaces cannot be avoided
in the edge removal step of the PC algorithm, consider the graph $G$ in
Figure~\ref{fig:ex}, taken from \cite[Example 4.8]{Uhler14}.
In the beginning, the PC algorithm finds
all nonconditional independencies (so with $S=\emptyset$), and hence
removes the edge $1-2$ to arrive at the graph $H$ on the right. If the
algorithm next chooses to consider the edge $1-4$, then it will delete
this edge after finding that $1,4$ are independent given $3$. However,
by symmetry of $H$ it is equally likely that it will first consider
the edge $1-3$. 

In \cite{Uhler14} it is shown that the partial correlation $f$ with
$i_0=1,j_0=3$ and $S=\{4\}$ has a singular hypersurface $H_f \subseteq
\RR^D$ and that the corresponding $\Tube(\lambda)$ of bad parameter
values is fatter.
\hfill $\clubsuit$
\end{ex}

\begin{ex} \label{ex:volineq}
In addition to the study of correlation hypersurfaces, \cite{Uhler14}
discusses mathematical interpretations of existing heuristics in
statistics. In particular, \cite[Problem 6.2]{Uhler14} discusses a volume
inequality that would confirm the belief that {\em ``collider-stratification
bias tends to attenuate when it arises from more extended paths''}. 
\begin{figure}
\includegraphics[scale=.7]{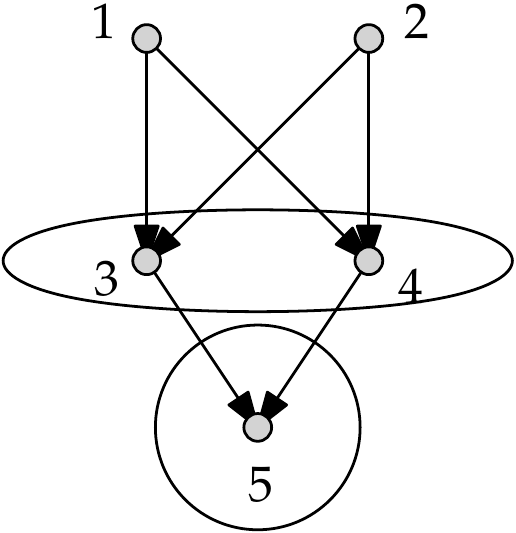}
\caption{The graph from Example~\ref{ex:volineq}.}
\label{fig:volineq}
\end{figure}
In particular, in the situation of Figure~\ref{fig:volineq}, their
conjecture says that
\[ 
\Vol(\{\lambda : |\corr(1,2|5)| \leq \lambda\})
\geq 
\Vol\{\lambda : |\corr(1,2|3,4)| \leq \lambda\}.
\] 
The paper does not explicitly say with respect to which measure $\Vol$
is defined here. If it is supposed to be true for all measures, then
the above is equivalent to
$\corr(1,2|5) \leq \corr(1,2|3,4).$
This is certainly not true in general: taking 
\[ a_{13}=-3, a_{14}=-2, a_{23}=8, a_{24}=10, a_{35}=2,
a_{45}=0 \]
yields $\corr(1,2|5)^2=1024/1189>88/105=\corr(1,2|3,4)^2.$
So formulating this statistical belief as a precise mathematical
conjecture remains a challenge. \hfill $\clubsuit$
\end{ex}


\begin{thebibliography}{{Wri}34}

\bibitem[DSS09]{Drton09}
Mathias Drton, Bernd Sturmfels, and Seth Sullivant.
\newblock {\em Lectures on Algebraic Statistics}.
\newblock Oberwolfach Seminars 39. Birkh{\"a}user, Basel, 2009.

\bibitem[DST13]{Draisma12h}
Jan Draisma, Seth Sullivant, and Kelli Talaska.
\newblock Positivity for {G}aussian graphical models.
\newblock {\em Adv.~Appl.~Math.}, 50(5):661--674, 2013.

\bibitem[GV85]{Gessel85}
Ira {Gessel} and G\'erard {Viennot}.
\newblock {Binomial determinants, paths, and hook length formulae.}
\newblock {\em {Adv. Math.}}, 58:300--321, 1985.

\bibitem[LUSB14]{Uhler14}
Shaowei Lin, Caroline Uhler, Bernd Sturmfels, and Peter B\"uhlmann.
\newblock Hypersurfaces and their singularities in partial correlation testing.
\newblock {\em Found. Comp. Math.}, 14:1079--1116, 2014.

\bibitem[SGS01]{Spirtes01}
Peter {Spirtes}, Clark {Glymour}, and Richard {Scheines}.
\newblock {\em {Causation, prediction, and search. With additional material by
  David Heckerman, Christopher Meek, Gregory F. Cooper and Thomas Richardson.
  2nd ed.}}
\newblock Cambridge, MA: MIT Press, 2nd ed. edition, 2001.

\bibitem[STD10]{Draisma09d}
Seth Sullivant, Kelli Talaska, and Jan Draisma.
\newblock Trek separation for {G}aussian graphical models.
\newblock {\em Ann.~Stat.}, 38(3):1665--1685, 2010.

\bibitem[{Wri}34]{Wright34}
S.~{Wright}.
\newblock {The method of path coefficients.}
\newblock {\em {Ann. Math. Stat.}}, 5:161--215, 1934.

\end{thebibliography}

\end{document}